\documentclass[a4paper,11pt]{amsart}

\usepackage{latexsym}
\usepackage{a4wide}
\usepackage{amscd}
\usepackage{graphics}
\usepackage{amsmath}
\usepackage{amssymb}
\usepackage{mathrsfs}

\newcommand{\R}{\mathbb R}
\newcommand{\C}{\mathbb C}

\newcommand{\Z}{\mathbb Z}

\newcommand{\sphere}[1]{\mathbb{S}^{#1}}
\newcommand{\e}{\mathrm e}
\newcommand{\abs}[1]{\left\vert #1 \right\vert}

\newcommand{\dom}{\textit{D}^{1,2}(\R^{N})}

\newtheorem{thm}{\bf Theorem}[section]      
\newtheorem{lem}[thm]{\bf Lemma}            
\newtheorem{prop}[thm]{\bf  Proposition}     
\newtheorem{defn}[thm]{\bf Definition}      
\newtheorem{rem}[thm]{\bf Remark}       

\begin{document}

\title[Biradial solutions to elliptic equations]{A note on the complete rotational invariance of biradial solutions to semilinear elliptic equations}
\author{L. Abatangelo} \author{S. Terracini}

\address{L. Abatangelo and S. Terracini: Dipartimento di Matematica e Applicazioni, Universit\`a di Milano Bicocca, Piazza Ateneo Nuovo, 1, 20126 Milano (Italy)}
\email{l.abatangelo@campus.unimib.it, susanna.terracini@unimib.it}



\date{\today}
\maketitle

\begin{abstract}
\noindent We investigate symmetry properties of solutions to
equations of the form
\[-\Delta u=\frac{a}{\abs{x}^2}u+f(\abs{x},u)\]
in $\R^N$ for $N\geq4$, with at most critical nonlinearities. By
using geometric arguments, we prove that solutions with low Morse
index (namely 0 or 1) and which are biradial (i.e. are invariant under the action of a toric group of rotations), are in fact completely radial. A similar result holds for the semilinear Laplace-Beltrami equations on the sphere. Furthermore, we show that the condition on the Morse index is sharp. Finally we apply the result in order to estimate best constants of Sobolev type inequalities with different symmetry constraints.
\end{abstract}

\section{Introduction and statement of the result}\label{sec:intro}
Let $x=(\xi,\zeta)\in\R^k\times\R^{N-k}$, with $k,N-k\geq 2$. A function $u\colon\R^N\to\R$ is termed \emph{biradial}
if it is invariant under the action of the subgroup $SO(k)\times SO(N-k)$ of the group of rotations, namely, if there exists $\varphi\colon \R^+\times\R^+\to\R$ such that $u(\xi,\zeta)=\varphi(|\xi|,|\zeta|)$.
Consider the equation
\begin{equation}\label{eq:main}
-\Delta u=\frac{a}{\abs{x}^2}u+f(\abs{x},u) \qquad \textrm{in $\R^N\setminus\{0\}$,}
\end{equation}
in this paper, we wonder under what circumstances it is possible to assert that a biradial solution to \eqref{eq:main} is actually radially symmetric.

This problem arises from \cite{Ter96}, where the following
symmetry breaking result is given for the critical nonlinearity
$f(\abs{x},u)=u^{(N+2)/(N-2)}$: if $a<0$ and $\abs{a}$ is sufficiently large,
there are at least two distinct positive solutions, one being
radially symmetric and the second not. These solution are obtained by
 minimization of the associated Rayleigh
quotient over functions possessing either the full radial symmetry or a
discrete group of symmetries, namely, for given $k\in\Z$, functions which are invariant under the
$\Z_k\times SO(N-2)$-action on $\dom$ given by
$$u(\xi,\zeta) \mapsto v(\xi,\zeta)=u\big(R\xi,T\zeta\big)\ ,$$
$T$ being any rotation of $\R^{N-2}$ and $R$ a fixed rotation of order $k$. Once proved that the infimum
taken over the $\Z_k\times SO(N-2)$-invariant functions is achieved,
by comparing its value with the infimum taken over the radial
functions, one deduces the occurrence of symmetry breaking (see also \cite{ATer10}).

In order to obtain multiplicity of solutions, the first attempt is to increase the order $k$ of the symmetry group and, eventually,
to let it diverge to infinity, finding in the limit a minimizer of the Rayleigh quotient over the biradial functions. Now, will all these  solutions be distinct and different from the radial one? When examining this question, we need to take into account the construction due to Ding  of an infinity of nontrivial biradial  solutions to the Lane--Emden equation with  critical nonlinearity (cfr \cite{D86}). In that case it is well known that there is a unique family of radially symmetric solutions, which are the global minimizers of the Rayleigh quotients, while in Ding's construction the nontrivial biradial solutions have a Morse index larger than $2$.

We recall the following definition:

\begin{defn}
 The (plain, radial, biradial) \emph{Morse index} of a solution $u$ is the dimension
of the maximal subspace of the space of (all, radial, biradial) functions of  $\mathcal C_0^\infty(\R^N\setminus\{0\})$ on which the quadratic form associated to the linearized equation at $u$  is negative definite.
\end{defn}

We stress it is rather a geometric definition, so it is independent from any spectral theory about the differential operator we are dealing with.

The recent literature indicates that, for general semilinear equations,
solutions having low Morse index do likely possess extra symmetries.
Following these ideas and questions, we investigated in particular
the biradial solutions with a low Morse index, and we are able to prove the
following

\begin{thm}\label{doublyradialth}
Let $u\in\dom$ be a biradial solution to
\begin{equation}\label{eqdr}
 -\Delta u=\frac{a}{\abs{x}^2}u+f(\abs{x},u)
\end{equation}
with $a>-\big(\frac{N-2}{2}\big)^{2}$ and $f:\R^N\times\R\rightarrow\R$ being a Carath\'{e}odory function, $C^1$ with respect to $z$, such that it satisfies the growth restriction
$$\abs{f'_y(\abs{x},y)} \leq C(1+\abs{y}^{2^*-2})$$
for a.e. $x\in\R^N$ and for all $y\in\C$.

If the solution $u$ has biradial Morse index $m(u)\leq1$, then $u$ is radially symmetric.
\end{thm}

An analogous  result also holds for bounded domains having rotational symmetry,
and for elliptic equations on the sphere. The following result holds in any dimension $N\geq 3$:

\begin{thm}\label{doublyradialthmsphere}
 Let $f\in\mathcal C^1(\R;\R)$: if $u\in\mathcal C^2(\sphere{N})$ is a biradial solution to
 $$-\Delta_{\sphere{N}}u=f(u)$$
with $N\geq 3$, and it has biradial Morse index $m(v)\leq1$, then $u$ is constant on the sphere $\sphere{N}$.
\end{thm}

The paper is organized as follows: the next section is devoted to
introduce the main tools and facts which will play a key role
within the proof; in section 3 we present the proofs of Theorems
\ref{doublyradialth} and \ref{doublyradialthmsphere} splitting it according to solutions' Morse
index. In section 4 we give applications to the estimate of the best constants in some Sobolev type embeddings with 
symmetries. Finally section 5 is devoted to the discussion of the sharpness of the Theorems with respect to the Morse index.

\section{Preliminaries}

Here we start the proof of Theorem \ref{doublyradialth}. For the sake of simplicity, we will work in dimension $N=4$. We devote
the last part of the proof to discuss the validity of the result in higher
dimensions.


Let us consider the following three orthogonal vector fields in $\R^4$:
\begin{eqnarray*}
 X_1= \left[
 \begin{array}{l}
 \phantom{-}x_2 \\
 -x_1 \\
 \phantom{-}x_4 \\
 -x_3
 \end{array}
 \right],\quad
X_2=\left[
 \begin{array}{l}
 \phantom{-}x_4 \\
 \phantom{-}x_3 \\
 -x_2 \\
 -x_1
 \end{array}
 \right],\quad
X_3=\left[
 \begin{array}{l}
 -x_3 \\
 \phantom{-}x_4 \\
 \phantom{-}x_1 \\
 -x_2
 \end{array}
 \right]
\end{eqnarray*}

 The related derivatives
\[
 w_i = \nabla u \cdot X_i\;,\qquad i=1,2,3
\]
represent the infinitesimal variations of the function $u$
along the flows of the vector fields $X_i$
respectively. As the equation is invariant under the action of such flows,
these directional derivatives are solutions to the linearized equation
\begin{equation}\label{linearizedeq}
 -\Delta w-\frac{a}{\abs{x}^2}w=f'_y(\abs{x},u)w\ .
\end{equation}
We can associate the  singular differential operator
\begin{equation}\label{linearizedop}
 L_u w=-\Delta w-\frac{a}{\abs{x}^2}w-f'_y(\abs{x},u)w\ .
\end{equation}

\begin{rem}\label{rem1}
The vector space of $\{X_1,X_2,X_3\}$ generates the whole group of
infinitesimal rotations on the sphere of $\R^4$,
which can be structured as a 3-dimensional manifold.
In order to prove Theorem \ref{doublyradialth} it will be sufficient
to show that every $w_i\equiv0$.
\end{rem}

Obviously, we have $w_1\equiv0$ because the vector field $X_1$
generates the rotations under which the function $u$ is invariant for.  Let us fix polar coordinates
\begin{equation}\label{polarcoordinates}
\begin{cases}
x_1=r_1\cos\theta_1 \\
x_2=r_1\sin\theta_1
\end{cases}\qquad\qquad
\begin{cases}
x_3=r_2\cos\theta_2  \\
x_4=r_2\sin\theta_2\;;
\end{cases}
\end{equation}
we have $r_1=\sqrt{{x_1}^2 + {x_2}^2}$, $r_2=\sqrt{{x_3}^2 + {x_4}^2}$ and
$\theta_1=\arctan\frac{x_2}{x_1}$, $\theta_2=\arctan\frac{x_4}{x_3}$.

Therefore, since $u$  is biradial, we have
\[
w_i=\nabla u\cdot X_i=w(r_1,r_2)z_i(\theta_1,\theta_2)\;, \qquad i=2,3\;,
\]

where

\[
w(r_1,r_2)=\frac{\partial u}{\partial r_1}r_2-\frac{\partial u}{\partial
r_2}r_1,\qquad
z_2 = \sin(\theta_1+\theta_2),\qquad
z_3 = -\cos(\theta_1+\theta_2).
\]

\begin{rem}\label{rem2}
According to Remark \ref{rem1}, to our aim it will be sufficient
to prove that $w\equiv0$.
\end{rem}

We now focus our attention on a few fundamental properties of
the functions $w_i$. At first, as the $z_i$'s are spherical harmonics
and depend on the angles
$\theta_1$ and $\theta_2$ only,
we have
\begin{equation}\label{eq:z}
-\Delta z_i=\left(\frac{1}{r_1^2}+\frac{1}{r_2^2}\right)z_i \qquad
\textrm{for $i=2,3$.}
\end{equation}
Joining this with the linearized equation \eqref{linearizedeq}
solved by the $w_i$'s, we obtain the equation for $w$.
\begin{prop}\label{eqw}
The function $w$ is a solution to the following equation
\begin{equation}\label{linearizedeq2}
-\Delta w - \frac{a}{\abs{x}^2}w
-f'_y(\abs{x},u)w+\left(\frac{1}{r_1^2}+\frac{1}{r_2^2}\right)w=0.
\end{equation}
\end{prop}
\begin{proof} It holds that
\[
f'_y(\abs{x},u)w_i
=-\Delta(wz_i)- \frac{a}{\abs{x}^2}wz_i=-\Delta w\,z_i-\nabla
w\cdot\nabla z_i-w\Delta z_i- \frac{a}{\abs{x}^2}wz_i.
\]
Since $\nabla w\cdot\nabla z_i=0$,
thanks to \eqref{eq:z}, this becomes
$$-\Delta w \,z_i-
\frac{a}{\abs{x}^2}wz_i=f'_y(\abs{x},u)w\,z_i+w\Delta
z_i=f'_y(\abs{x},u)w\,z_i-\left(\frac{1}{r_1^2}+\frac{1}{r_2^2}\right)wz_i$$
that is
\[
z_i\,\bigg\{-\Delta w - \frac{a}{\abs{x}^2}w - f'_y(\abs{x},u)w +
\Big(\frac{1}{r_1^2} + \frac{1}{r_2^2} \Big)\bigg\} = 0.
\]
Last, multiplying by $z_i$ and summing for $i=1,2$ we obtain the desired equation.  
\end{proof}

\section{Proofs}\label{sec:proofs}

We will split the argument according to the Morse index of solution $u$:
we denote it by $m(u)$.

In order to complete our proof, we need a couple of preliminary results: the
first one is about the asymptotics of the solution and is contained in \cite{FMT08}.

\begin{lem}(\cite{FMT08})\label{fft}
Under the assumptions of Theorem \ref{doublyradialth}, let $u$ be any solution to \eqref{eq:main}. Then 
the following asymptotics hold
\begin{eqnarray}
u(x) \sim \abs{x}^\gamma \psi(\dfrac{x}{\abs{x}}) \qquad \textrm{for $\abs{x}\ll1$}\label{asintotico fmt origine}\\
u(x) \sim \abs{x}^\delta \psi(\dfrac{x}{\abs{x}}) \qquad \textrm{for $\abs{x}\gg1$}\label{asintotico fmt infinito}
\end{eqnarray}
where
$\gamma=\gamma(a,N)=-\frac{N-2}{2}+\sqrt{\left(\frac{N-2}{2}\right)^2+\mu}$,
$\delta=\delta(a,N)=-\frac{N-2}{2}-\sqrt{\left(\frac{N-2}{2}\right)^2+\mu}$
and $\mu=\mu(a,N)$ is one of the eigenvalues of
$-\Delta_{\sphere{N-1}}-a$ on $\sphere{N-1}$, and $\psi$ one of its
related eigenfunctions.
\end{lem}
This turns out to be the key for proving the following result.
\begin{lem}
The function $\left(\frac{1}{r_1^2}+\frac{1}{r_2^2}\right)w^2$
is $L^1$-integrable on $\R^N$.
\end{lem}
\begin{proof}
Since $w(r_1,r_2)=\frac{\partial u}{\partial
r_1}r_2-\frac{\partial u}{\partial r_2}r_1$, we first observe that by regularity of $u$
outside the origin and its radial symmetry, the functions
\[\frac 1{r_i}\frac{\partial u}{\partial
r_i}
\]
$i=1,2$, are continuous outside the origin. Next we remark that
\begin{multline*}
\frac14\left(\frac{1}{r_1^2}+\frac{1}{r_2^2}\right)w^2\leq\lefteqn{\left(\frac{1}{r_1^2}+\frac{1}{r_2^2}\right)\left\{\left(\frac{\partial
u}{\partial r_1}\right)^2r_2^2+\left(\frac{\partial u}{\partial
r_2}\right)^2r_1^2\right\}}\\
 = \left(\frac{r_2}{r_1}\right)^2\left(\frac{\partial u}{\partial
r_1}\right)^2 + \left(\frac{r_1}{r_2}\right)^2\left(\frac{\partial
u}{\partial r_2}\right)^2 + \left(\frac{\partial u}{\partial
r_1}\right)^2 + \left(\frac{\partial u}{\partial r_2}\right)^2.
\end{multline*}
The integrability of the last two terms is a straightforward
consequence of $u\in\dom$. In order to study the other two terms,
let us focus our attention in a ball around the origin, namely
$B_1(0)$, so that $r_1^2+r_2^2\leq1$. Then
$$\left(\frac{\partial u}{\partial r_1}\right)^2\left(\frac{r_2}{r_1}\right)^2
\leq \frac{1}{r_1^2}\left(\frac{\partial u}{\partial
r_1}\right)^2-\left(\frac{\partial u}{\partial r_1}\right)^2,$$ so
that the question of integrability is restricted to the first term.
From Lemma \ref{fft}, Equation \eqref{asintotico fmt origine} we know $u\sim r^\gamma\psi(r_1,r_2)
=(r_1^2+r_2^2)^{\gamma/2}\psi(r_1,r_2)$, from which
$$\dfrac{\partial u}{\partial r_1}\sim\psi(r_1,r_2)\gamma(r_1^2+r_2^2)^{\gamma/2-1}r_1
+ (r_1^2+r_2^2)^{\gamma/2}\dfrac{\partial \psi}{\partial r_1}.$$
So we are lead to consider the integrability of $\int_{B_1(0)}\frac{1}{r_1^2}\left(\frac{\partial\psi}{\partial r_1}\right)^2$.
Additionally we know that $\psi$ is the restriction on the sphere of a harmonic polynomial, then
it is analytic and its Taylor's expansion
is a polynomial whose degree 1 terms vanish, since it is a function of the only
variables $r_1$ and $r_2$. Then $\left(\frac{\partial \psi}{\partial r_1}\right)^2 \sim r_1^2$,
which provides the sought integrability.

For what concerns the integrability at infinity, it is sufficient to show that the
terms of type $\frac{r_2^2}{r_1^2}\left(\frac{\partial u}{\partial r_1}\right)^2$
are in $L^1(\R^N)$. We have
$$\frac{r_2^2}{r_1^2}\left(\frac{\partial u}{\partial r_1}\right)^2
\leq 2\left\{ (r_1^2+r_2^2)^{\delta-2}r_2^2\psi^2
+ \dfrac{r_2^2}{r_1^2}(r_1^2+r_2^2)^\delta\left(\dfrac{\partial\psi}{\partial r_1}\right)^2\right\}$$
and exploiting equation \eqref{asintotico fmt infinito}, the expression of the exponent $\delta$ provides the sought integrability. 
\end{proof}

In the following we consider the cut-off function defined as $\eta(r_1,r_2)=\eta_1(r_1)\eta_2(r_2)$ where
\begin{eqnarray*}
 \eta_1(r_1) = \left\{
\begin{array}{ll}
 \dfrac{1}{\log{(R_2/R_1)}}\log{r_1/R_1} \quad &\textrm{for $R_1\leq r_1\leq R_2$} \\
 1 \quad &\textrm{for $R_2\leq r_1\leq R_3$} \\
 1 - \dfrac{1}{\log{(R_4/R_3)}}\log{r_1/R_3} \quad &\textrm{for $R_3\leq r_1\leq R_4$}\\
 0 \quad &\textrm{elsewhere,} \\
\end{array}\right.
\end{eqnarray*}
$\eta_2$ being defined similarly. Given the special form of $\eta$, we note
$\abs{\nabla \eta}^2 \leq \abs{\nabla\eta_1}^2 +
\abs{\nabla\eta_2}^2$, that is
\begin{eqnarray*}
\abs{\nabla \eta}^2 \leq \dfrac{1}{\log^2{R_2/R_1}}\left(\frac{1}{r_1^2} + \frac{1}{r_2^2} \right) \qquad \textrm{for $R_1 \leq r_1\,,r_2 \leq R_2$} \\
\end{eqnarray*}
and analogously for $R_3 \leq r_1\,,r_2 \leq R_4$. Thus, we have
\begin{equation}\label{stima_gradiente_eta}
\abs{\nabla\eta}^2 \leq 3\left(\dfrac{1}{\log^2
R_4/R_3}+\dfrac{1}{\log^2
R_2/R_1}\right)\left(\dfrac{1}{r_1^2}+\dfrac{1}{r_2^2}\right).
\end{equation}

\begin{lem}\label{w+w-}
There is a suitable choice of the parameters $R_1$, $R_2$, $R_3$ and
$R_4$ such that the quadratic form associated to the operator
\eqref{linearizedop} is negative definite both on $\eta\,w^+$ and
$\eta\,w^-$.
\end{lem}
\begin{proof}
Let us fix $\varepsilon>0$ small and choose $R_1=\varepsilon^2$,
$R_2=\varepsilon$ and $R_3=\varepsilon^{-1}$,
$R_4=\varepsilon^{-2}$. We multiply equation \eqref{linearizedeq2}
by $\eta^2 w^+$ and integrate by parts. We obtain
\begin{eqnarray}\label{lemma cut off positivita autofunz}
 \lefteqn{\int_{\R^N} \abs{\nabla(\eta^2 w^+)}^2 - \dfrac{a}{\abs{x}^2}(\eta^2 w^+)^2 - f_y'(\abs{x},u)\eta^2(w^+)^2 } \nonumber \\
&&= \int_{\R^N} \abs{\nabla\eta}^2(w^+)^2 - \int_{\R^N}\left(\frac{1}{r_1^2}+\frac{1}{r_2^2}\right)\eta^2(w^+)^2\,.
\end{eqnarray}
If $\varepsilon$ is small enough, the second term in \eqref{lemma cut
off positivita autofunz} is far away from zero, or rather, it is
quite close to
$\int_{\R^N}\left(\frac{1}{r_1^2}+\frac{1}{r_2^2}\right)(w^+)^2$,
say for instance
$$\int_{\R^N}\left(\frac{1}{r_1^2}+\frac{1}{r_2^2}\right)\eta^2(w^+)^2
>\frac{1}{2}\int_{\R^N}\left(\frac{1}{r_1^2}+\frac{1}{r_2^2}\right)(w^+)^2.$$
On the other hand, the first term in \eqref{lemma cut off positivita
autofunz} can be made very small with respect to\\
$\int_{\R^N}\left(\frac{1}{r_1^2}+\frac{1}{r_2^2}\right)(w^+)^2$,
since from \eqref{stima_gradiente_eta}
$$\int_{\R^N}\abs{\nabla\eta}^2(w^+)^2 \leq
\dfrac{6}{\log^2\varepsilon}\int_{\R^N}\left(\frac{1}{r_1^2}+\frac{1}{r_2^2}\right)(w^+)^2,$$
so that \eqref{lemma cut off positivita autofunz} is seen to be
negative.

Repeating the same argument multiplying by $\eta^2 w^-$ we reach the same conclusion.
\end{proof}

\noindent{\bf First case: Morse index $m(u)=0$.} In this case Lemma \ref{w+w-}
clearly contradicts the hypothesis $m(u)=0$,
unless $w^+ = w^- \equiv 0$, that is the only stable solution to \eqref{linearizedeq2}
is the trivial one. 

\noindent{\bf Second case: Morse index $m(u)=1$.}
 { \it In this case we infer that $w$ has constant sign}, say positive,
and therefore $w>0$ for $r_1>0$ and $r_2>0$ by the  Strong Maximum Principle. 
 Now we show a contradiction.  Consider a
vector field of the form $\alpha X_2+\beta
  X_3$. Along this vector field, choosing $\alpha=\cos\gamma$ and $\beta=\sin\gamma$,
the derivative of $u$ is
  \begin{equation*}
  \nabla u\cdot (\alpha X_2+\beta
  X_3)=\alpha w_2+\beta
  w_3=w\big(\alpha\sin(\theta_1+\theta_2)-\beta\cos(\theta_1+\theta_2)\big)=-w\sin(\theta_1+\theta_2-\gamma).
  \end{equation*}
  Now we turn to the directional derivative of  $\theta_1+\theta_2$ along
the vector field $\alpha X_2 + \beta X_3$.
  Using the polar coordinates \eqref{polarcoordinates}, it results
  \begin{equation*}
  \theta_1=\arctan\frac{x_2}{x_1}\;,\qquad
  \theta_2=\arctan\frac{x_4}{x_3};
  \end{equation*}
  so that checking the motion along $X_2$ we have
  \begin{eqnarray*}
  \nabla \theta_1 \cdot X_2 &=& \frac{x_3x_1-x_2x_4}{x_1^2+x_2^2}
  =\frac{r_2}{r_1}(\cos\theta_1\cos\theta_2-\sin\theta_1\sin\theta_2)
  =\frac{r_2}{r_1}\cos(\theta_1+\theta_2)\\
  \nabla \theta_2 \cdot X_2 &=& \frac{-x_3x_1+x_2x_4}{x_3^2+x_4^2}
  =\frac{r_1}{r_2}(-\cos\theta_1\cos\theta_2+\sin\theta_1\sin\theta_2)
  =-\frac{r_1}{r_2}\cos(\theta_1+\theta_2),
  \end{eqnarray*}
  whereas along $X_3$
  \begin{eqnarray*}
  \nabla \theta_1 \cdot X_3 &=& \frac{x_4x_1+x_2x_3}{x_1^2+x_2^2}
  =\frac{r_2}{r_1}(\cos\theta_1\sin\theta_2 + \sin\theta_1\cos\theta_2)
  =\frac{r_2}{r_1}\sin(\theta_1+\theta_2)\\
  \nabla \theta_2 \cdot X_3 &=& \frac{-x_2x_3-x_1x_4}{x_3^2+x_4^2}
  =\frac{r_1}{r_2}(-\sin\theta_1\cos\theta_2-\cos\theta_1\sin\theta_2)
  =-\frac{r_1}{r_2}\sin(\theta_1+\theta_2);
  \end{eqnarray*}
and finally we obtain
  \begin{multline}
  \nabla(\theta_1+\theta_2)\cdot(\alpha X_2+\beta X_3)=
\left(\frac{r_2}{r_1}-\frac{r_1}{r_2}\right)\big(\alpha\cos(\theta_1+\theta_2)+
\beta\sin(\theta_1+\theta_2)\big)\nonumber \nonumber \\ \label{eqtheta}
  =\left(\frac{r_2}{r_1}-\frac{r_1}{r_2}\right)\cos(\theta_1+\theta_2-\gamma).
  \end{multline}
  Now we are in good position to conclude.
  For a given point $\overline{x}$ of the sphere -
located by angles $\overline{\theta_1}$ and $\overline{\theta_2}$, we choose
  $\gamma=\gamma(\overline{x})=\overline{\theta_1}+\overline{\theta_2}-\pi/2$,
  so that the quantity $\theta_1+\theta_2$ is at rest for the associated vector
field $\cos\gamma X_2+\sin\gamma X_3$. With
  this choice the function $u$ is monotone along the flow $\alpha X_2+\beta
  X_3$ since $\dot{u}=-w\sin(\theta_1+\theta_2-\gamma)=-w$ and the sign of
  $w$ is constant by the previous discussion. Since the trajectory of
  the flow is a circle, we will reach again the initial point in
  finite time, but with a strictly smaller value of $u$
(if we consider the first eigenfunction $w$ positive). This is
  clearly a contradiction.

\noindent \textit{Generalization to higher dimensions.} In dimension
 $N\geq5$ the argument is very similar. Relabeling we may always assume $u= u(\rho_1,\rho_2)$, where we have fixed the notation
 $\rho_1=\abs{\xi}$ and $\rho_2=\abs{\zeta}$, while
 $|x|=\sqrt{\abs{\xi}^2+\abs{\zeta}^2}$, being $x=(\xi,\zeta)\in \R^k\times\R^{N-k}$.
Now we repeat the
argument performed in the 4-dimensional space with respect to the variables $x_{k-1}$, $x_{k}$, $x_{k+1}$, $x_{k+2}$, considering the vector fields
with those same four components as above and the other ones being zero. Hence we define $r_1=\sqrt{x_{k-1}^2+x_{k}^2}$ and $r_2=\sqrt{x_{k+1}^2+x_{k+2}^2}$. When discussing the integrability properties, it can be worthwhile noticing that
\[\dfrac1{r_i}\dfrac{\partial u}{\partial r_i}=\dfrac1{\rho_i}\dfrac{\partial u}{\partial \rho_i}\;.\]
Arguing as above, we  can prove that the solution $u$ is actually
radial with respect to those four variables. We can imagine to
iterate this proceeding for every hyperplane whose rotations the function $u$ is supposed not to be invariant for.
Finally, it follows that $u$ is radial in $\R^N$.

\smallskip

{\it Proof of Theorem \ref{doublyradialthmsphere}.}
Since
now $v$ is a function defined over $\sphere{N}$,
recalling the Laplace operator in polar coordinates
 $$\Delta_{\R^{N+1}}=\partial_r^2+\frac{N}{r}\partial_r+\frac{1}{r^2}\Delta_{\sphere{N}},$$
we define $\widetilde{v}(x)=v(y)$ for
$x\in(-\varepsilon,\,\varepsilon) \times \sphere{N}$, so that
$\Delta_{\sphere{N}}v=\Delta_{\R^{N+1}}\widetilde{v}$. At first, let us suppose $N=3$. Obviously,
since $v$ is invariant with respect to the group $O(2)\times
O(2)$, so is $\widetilde{v}$.

Following the same argument in the proof of Theorem \ref{doublyradialth}, we wish to prove the vanishing of
$\widetilde{w}=\frac{\partial\widetilde{v}}{\partial
r_1}r_2-\frac{\partial\widetilde{v}}{\partial r_2}r_1$. On the other
hand, being $\widetilde{v}$ homogenous of degree 0, $\widetilde{w}$
is homogenuos of degree 0 too (it can be proved by differentiating
identity $\widetilde{v}(x)=\widetilde{v}(\lambda x)$), then the $w$
associated with $v$ is nothing else that $\widetilde{w}$ restricted on
the sphere $\sphere{N}$. Therefore
$\Delta_{\R^{N+1}}\widetilde{w}=\Delta_{\sphere{N}}w$, and following
the proof of Proposition \ref{eqw} we see $w$ is a solution to
$$ -\Delta_{\sphere{N}}w-f'(v)w+\left(\frac{1}{r_1^2}+\frac{1}{r_2^2}\right)w=0,$$
analogous to equation \eqref{linearizedeq2}. The rest of the proof fits also in this case. \qed

\section{An application to best Sobolev constants with symmetries}


Solutions to the critical exponent equation
\begin{equation}\label{magnetic eq}
-\Delta u = \frac{a}{|x|^2}u + \abs{u}^{2^*-2}u
\end{equation}
are related to extremals of Sobolev inequalities (cfr \cite{Ter96}).
To our purposes, the functions $u$ will be complex-valued and $a\in(-\infty,(N-2)^2/4)$.
Then, thanks to Hardy inequality,  an equivalent norm on $D^{1,2}(\R^N)$ is
$$\left(\int_{\R^N}\abs{\nabla u}^2-a\frac{\abs{u}^2}{|x|^2}\right)^{1/2},$$
hence we can  seek solutions to \eqref{magnetic
eq} as extremals of the Sobolev quotient associated with this norm on different symmetric spaces . 

The whole
group of rotations $SO(2)\times SO(N-2)$ induces the
following action on $D^{1,2}(\R^N;\C)$:
$$u(\xi,\zeta) \mapsto R^{-m} u(R\xi,T\zeta)$$
for $m\in\Z$ fixed. We denote, as usual, $D^{1,2}_{rad}(\R^N)$ and $D^{1,2}_{\text{birad}}(\R^N)$ the subspaces of real or complex radial and biradial functions. 
Moreover, let $k$ and $m$ be fixed integers; for a given rotation $R\in SO(2)$ of order $k$, we consider the space of symmetric functions
\[
D^{1,2}_{R,k,m}(\R^N;\C):=\{u\in D^{1,2}(\R^N;\C) : \ u(R\xi,T\zeta)=R^m u(\xi,\zeta), \forall\,\ T\in SO(N-2)\}.\]
This is of course a proper subspace of
\[
D^{1,2}_{\text{birad},m}(\R^N;\C):=\{u\in D^{1,2}(\R^N;\C) : \ u(S\xi,T\zeta)=S^m u(\xi,\zeta), \forall\,\ (S,T)\in SO(2)\times SO(N-2)\}.\]
 Note this last space coincides with the usual space of biradial solution once $m=0$.


Thanks to its rotational invariance, for any choice of the above spaces $D^{1,2}_{*}(\R^N;\C)$, solutions to the minimization problem
\begin{equation}\label{magnetic min problem}
 \inf_{u\in D^{1,2}_{*}(\R^N;\C) \atop u\neq 0}\frac{\displaystyle\int_{\R^{N}}\abs{\nabla u}^2 - a\frac{\abs{u}^2}{\abs{x}^2}}{\displaystyle\left(\int_{\R^{N}}\abs{u}^{2^*}\right)^{2/2^*}}
\end{equation}
 are in fact solutions to equation \eqref{magnetic eq}.

The minimization of the Sobolev quotient over the space of radial functions follows from a nowadays standard compactness argument; in addition, see for instance \cite{Ter96}, we have:
\[\inf_{u\in D^{1,2}_{rad}(\R^N;\C)\atop u\neq 0}\frac{\displaystyle \int_{\R^{N}}\abs{\nabla u}^2
- a\frac{\abs{u}^2}{\abs{x}^2}}{\displaystyle\left(\int_{\R^{N}}\abs{u}^{2^*}\right)^{2/2^*}}= S \Big(1-a\frac{4}{(N-2)^2}\Big)
\]
where $S$ denote the best constant for the standard Sobolev embedding.
Moreover, generalizing the results in \cite{CS10} in higher dimensions (see also \cite{ATer10}), one
can easily prove existence of minimizers of the Sobolev quotient \eqref{magnetic min problem} in the spaces $D^{1,2}_{\text{birad},m}(\R^N;\C)$, for any choice of the integer $m$.

At first, let us consider the  case $m=0$. Then it is easily checked that the minimizers can be chosen to be real valued and that the corresponding solution to \eqref{magnetic eq} have biradial Morse index exactly one. Hence our Theorem \ref{doublyradialth} applies and such biradial solutions are in fact fully radially symmetric, and therefore the infimum on the biradial space equals that on the radial. Now, let us turn to the case $m\neq 0$. We remark that elements of the space $D^{1,2}_{\text{birad},m}(\R^N;\C)$ have the form $u((\xi,\zeta)=\rho(\abs{\xi},\abs{\zeta})\e^{i m
\theta(\xi)}$, where $\theta(\xi)=\arg(\xi)$, so that
$$\abs{\nabla u}^2=\abs{\nabla\rho}^2 + \rho^2\abs{m\nabla\theta}^2 = \abs{\nabla\rho}^2 + m^2 \frac{\rho^2}{\abs{\xi}^2}.$$

Then the following chain of inequalities holds:
\begin{multline*}
\min_{u\in D^{1,2}_{\text{birad},m}(\R^N;\C) \atop u\neq 0}\frac{\displaystyle\int_{\R^{N}}\abs{\nabla u}^2
- a\frac{\abs{u}^2}{\abs{x}^2}}{\displaystyle\left(\int_{\R^{N}}\abs{u}^{2^*}\right)^{2/2^*}}
 = \min_{\rho\in D^{1,2}_{\text{birad}}(\R^N,\R) \atop \rho\neq 0} \frac{\displaystyle\int_{\R^N}\abs{\nabla\rho}^2
+ m^2 \frac{\rho^2}{\abs{\xi}^2}-a\frac{\rho^2}{\abs{x}^2}}{\displaystyle\left(\int_{\R^{N}}\rho^{2^*}\right)^{2/2^*}} \\
 >\min_{\rho\in D^{1,2}_{\text{birad}}(\R^N;\R) \atop \rho\neq 0} \frac{\displaystyle\int_{\R^N}\abs{\nabla\rho}^2
+ (m^2-a) \frac{\rho^2}{\abs{x}^2}}{\displaystyle\left(\int_{\R^{N}}\rho^{2^*}\right)^{2/2^*}}
 = \min_{\rho\in D^{1,2}_{rad}(\R^N;\R) \atop \rho\neq 0} \frac{\displaystyle\int_{\R^N}\abs{\nabla\rho}^2
+ (m^2-a) \frac{\rho^2}{\abs{x}^2}}{\displaystyle\left(\int_{\R^{N}}\rho^{2^*}\right)^{2/2^*}}\\
= S \Big(1+\frac{4(m^2-a)}{(N-2)^2}\Big)
\end{multline*}
where we have used $\abs{\xi}\leq\abs{x}$; the intermediate line follows again from Theorem \ref{doublyradialth}, and the last from \cite{Ter96}.
Then, this argument states a very useful lower bound (see \cite{ATer10}) to the minima problems \eqref{magnetic min problem}. Indeed, it allows us to compare the
infimum over the space of $D^{1,2}_{R,k,m}(\R^N;\C)$ with that  on $D^{1,2}_{\text{birad},m}(\R^N;\C)$, and to prove the occurrence of symmetry breaking in some circumstances. In fact it has been proven (see \cite{ATer10}) that, for large enough $k$, the first minimum is achieved and less that $k^{2/N}S$, while the latter increases with $\abs{a}$ and $m$. Symmetry breaking holds whenever it can be shown that $1+\frac{4(m^2-a)}{(N-2)^2}>k^{2/N}$ for appopriate choices of the parameters.

\section{Optimality with respect to the Morse index}

We want to stress our results Theorem \ref{doublyradialth} and \ref{doublyradialthmsphere} are sharp
with respect to the Morse index. By that, we mean that doubly radial solutions with Morse index greater or equal to 2,
need not to be completely radial.

To prove this, we will take advantage from a result proved by Ding in
\cite{D86} in such a way which will be clear later. The quoted paper by Ding has
to do with solutions to a related equation on $\sphere{N}$, for
this reason we state first some connections between these two
environments.

\subsection{Conformally equivariant equations}

We recall a general fact cited in \cite{D86} about elliptic
equations on Riemannian manifolds.

\begin{lem}\label{lemma equazioni}
Let $(M,g)$ and $(N,h)$ two Riemannian manifolds of dimensions
$N\geq3$. Suppose there is a conformal diffeomorphism
$f:\,M\rightarrow N$, that is $f^*h=\varphi^{2^*-2}g$ for some
positive $\varphi\in C^\infty(M)$. The scalar curvatures of $(M,g)$
and $(N,h)$ are $R_g$ and $R_h$ respectively. Set the following
corresponding equations:
\begin{eqnarray}\label{ding}
-\Delta_g u + \frac{1}{4}\frac{N-2}{N-1}R_g(x)u&=&F(x,u) \label{eqM} \\
-\Delta_h v + \frac{1}{4}\frac{N-2}{N-1}R_h(y)v&=&[(\varphi\circ
f^{-1})(y)]^{-\frac{N+2}{N-2}}F(f^{-1}(y),(\varphi\circ
f^{-1})(y)v)\label{eqN}
\end{eqnarray}
where $F:M\times\R\rightarrow\R$ is smooth. Suppose $v$ is a
solution of \eqref{eqN}. Then $u=(v\circ f)\varphi$ is a solution of
\eqref{eqM} such that $\int_M\abs{u}^{2^*}\,dV_g=\int_N
\abs{v}^{2^*}\,dV_h$.
\end{lem}

We consider the inverse of the stereographic
projection $\pi:\sphere{N}\setminus\{p\}\rightarrow\R^N$. We denote
it by $\Phi=\pi^{-1}:\R^N\rightarrow\sphere{N}\setminus\{p\}$,
moreover $g_0$ will denote the standard metric on $\sphere{N}$ and
$\delta$ the standard one on $\R^N$.

The diffeomorphism $\Phi$ is conformal between the two manifolds,
since it results
$$ g\doteq\Phi^*g_0=\mu(x)^{\frac{4}{N-2}}\delta\ ,$$
where
$$\mu(x)=\left(\frac{2}{1+\abs{x}^2}\right)^{\frac{N-2}{2}}.$$

In addition, we point out the manifold $(\R^N,g)$ is \emph{the same}
as $(\sphere{N},g_0)$, in terms of diffeomorphic manifolds.

\bigskip


We recall the following
\begin{defn}
We define the \emph{conformal Laplacian} on a differentiable closed
manifold $(M,g)$ of dimension $N$ the operator
$$ L_g=-\Delta_g+\frac{N-2}{4(N-1)}R_g $$
where $\Delta_g$ denotes the standard Laplace-Beltrami operator on
$M$ and $R_g$ the scalar curvature of the manifold.
\end{defn}
Moreover, this operator has a simple transformation law under
a conformal change of metric, that is
$$ \textrm{if \quad $\widetilde{g}=\mu(x)^{\frac{4}{N-2}}g$\qquad
then\quad
$L_{\widetilde{g}}\,\cdot=\mu(x)^{-\frac{N+2}{N-2}}L_g\big(\mu(x)\cdot\big)$.}$$
In our case we are dealing with the same manifold $\R^N$ endowed
with the two metrics $\delta$, the standard one, and $g=\Phi^*g_0$.
Thus in our case we have
$$ L_\delta=-\Delta \qquad  L_g=-\Delta_g+\frac{1}{4}N(N-2) $$
so it is quite easy to check directly the correspondence between
the equations stated in Lemma \ref{lemma equazioni} by calculations.

\subsection{Proof of the optimality of Theorem \ref{doublyradialth} with respect to the Morse index}

In this section we discuss the optimality of Theorems \ref{doublyradialthmsphere} with respect to the solutions' Morse index.
First of all, we consider the the equation on the sphere $\sphere{N}$ related to \eqref{eqdr} through the weighted composition with
the stereographic projection $\pi$ as conformal
diffeomorphism from $\sphere{N}\setminus\{p\}$ onto $\R^N$: it is immediate to check that it is
$$ -\Delta_{\sphere{N}}v(y)+\frac{1}{4}N(N-2)v(y)=f(v(y)) \qquad y\in\sphere{N}.$$
In his paper \cite{D86}, Ding states the following result:
\begin{lem}\label{ding2}
 There exists a sequence $\{v_k\}$ of biradial solutions to the equation
\begin{equation}\label{eq:sphere}
 -\Delta_{\sphere{N}}v+\frac{1}{4}N(N-2)v=\abs{v}^{\frac{4}{N-2}}v \qquad v\in C^2(\sphere{N}) 
\end{equation}
such that $\int_{\sphere{N}}\abs{v_k}^{\frac{2N}{N-2}}dV\rightarrow\infty$ as $k\rightarrow\infty$.
\end{lem}
The choice of working in a space of biradial is motivated by the compact embedding of the space of $H^1$--biradial functions on the sphere into $L^{2N/(N-2)}$. In this way one can overcome the lack of compactness due to the presence of the critical exponent and prove the result as an application of the Ambrosetti-Rabinowitz symmetric Mountain Pass Theorem. 
We are interested in classifying the solutions according to their Morse index. We can state the following
\begin{lem}\label{ding3}
 Among the solutions $\{v_k\}$ in Lemma \ref{ding2} there is also a constant one, which is unique and corresponds to the minimum of Sobolev quotient. All the other biradial solutions have biradial Morse index at least 2, and there is at least one non constant biradial solution having Morse index exactly 2.
\end{lem}
\begin{proof}
 We can check directly there exists a unique constant solution:
$$ \frac{1}{4}N(N-2)c=c^{\frac{N+2}{N-2}} \quad \Longrightarrow \quad c=\left(\frac{1}{4}N(N-2)\right)^{\frac{N-2}{4}}\,.$$
which corresponds to the Talenti functions on the sphere (\cite{Talenti76}). We mean it is the image of the function $w(x)=\frac{\big(N(N-2)\big)^{\frac{N-2}{4}}}{\big(1+\abs{x}^2\big)^{\frac{N-2}{2}}}=\mu(x)c$ through the diffeomorphism $\pi^{-1}$ and
$$L_g c=\mu(x)^{-\frac{N+2}{N-2}}\Delta\big(\mu(x)c\big)\,.$$
Then it reaches the minimum of Sobolev quotient
$\inf_{v\neq0}\frac{\int_{\sphere{N}\abs{\nabla
v}^2}}{\left(\int_{\sphere{N}}\abs{v}^{\frac{2N}{N-2}}\right)^{2/2^*}}$,
and therefore it is quite simple to prove it is the mountain pass
solution, i.e. its (plain, radial, biradial) Morse index is $m(c)=1$. Now, thanks to Theorem \ref{doublyradialthmsphere}, every other biradial solution having biradial Morse index at most 1 is constant, hence all the other solutions have biradial Morse index at least 2. 
Now, it is well known that Talenti's solutions are unique among positive solutions of equation \eqref{eqdr} on $\R^N$, so  we can assert that the only biradial positive solutions of  \eqref{eq:sphere} are constant. On the other hand, it can be proven for example using Morse Theory in ordered Banach spaces (see \cite{BW96}), that the equation admits a biradial sign-changing solution having biradial Morse index at most 2. Hence there is a biradial solution of \eqref{eq:sphere} with Morse index exactly 2 which is not constant. \end{proof}

\end{document}